\documentclass[12pt, reqno]{amsart}
\usepackage{amsmath, amsthm, amscd, amsfonts, amssymb, graphicx, color}
\usepackage[colorlinks=true, linkcolor=blue,urlcolor=blue]{hyperref}
\newtheorem{theorem}{Theorem}[]
\newtheorem{lemma}[theorem]{Lemma}

\theoremstyle{definition}

\newtheorem{example}[theorem]{Example}

\theoremstyle{remark}
\newtheorem{remark}[theorem]{Remark}
\numberwithin{equation}{section}

\newcommand{\N}{\mbox{$\mathbb{N}$}}

\newcommand{\A}{\mbox{$\alpha$}}

\newcommand{\h}{Hurwitz }

\begin{document}
\setcounter{page}{1}

\title[Strongly primeness of skew Hurwitz polynomial rings]{strongly primeness of skew Hurwitz polynomial rings}

\author[Ali Shahidikia]{Ali Shahidikia}
\address{Department of Mathematics, Dezful branch, Islamic Azad University, Dezful, Iran}
\email{\textcolor[rgb]{0.00,0.00,0.84}{ali.Shahidikia@iaud.ac.ir \,\
a.shahidikia@gmail.com
}}


\begin{abstract}
For a ring $R$ and an endomorphism $\A$ of $R$, we  characterize the left and right strongly primeness of skew  \h polynomial ring $(hR,\alpha)$. 
\end{abstract}

\maketitle

\section*{}
The study of formal power series rings has garnered noteworthy attention and has been found to be essential in a multitude of fields, particularly in differential algebra. In \cite{kie}, Keigher introduced a variant of the formal power series ring and studied its categorical properties in detail. This ring was later named Hurwitz series ring.
There are many interesting applications of the Hurwitz series ring in differential algebra as Keigher showed in \cite{kieg,kiegh}.\\


Throughout this article, $R$ denotes an associative ring with unity, and $\A:R \rightarrow R$ is an endomorphism. We denote $(H(R),\A)$, or simply $(HR,\A)$,  the skew Hurwitz series ring over a ring $R$ whose elements are the functions $f:\N\rightarrow R$, where $\N$ is the set of all natural numbers, the addition is defined as usual and the multiplication given by 
\begin{equation*}
(f g)(n)= \sum_{k=0}^n \binom{n}{k} f(k) \A^k(g(n - k))
 \text{ for  all } n \in \N,
\end{equation*}
where $\binom{n}{k}$ is the binomial coefficient.\par


Define the mappings $h_n:\N\rightarrow R$, $n\geq1$ via $h_n(n-1)=1$ and $h_n(m)=0$ for each $ m \in \N\setminus\{n-1\}$ and $h'_r:\N\rightarrow R$, $r\in R$ via $h'_r(0)=r$ and $h'_r(n)=0$ for each $n\in \N\setminus\{0\}$. It can be easily shown that $h_{1}$ is the unity of $(HR,\A)$.

For $f\in (HR,\A)$, the support of $f$, denoted by $supp(f)$, is the set $\{i\in\N\mid  f(i)\not=0\}$. The minimal element in $supp(f)$ is denoted by $\Pi(f)$, and $\Delta(f)$ denotes the greatest element in $supp(f)$ if it exists. Define $R'=\{h'_r\mid r\in R\}$. $R'$ is a subring of $(HR,\A)$ which is isomorphism to $R$. For any nonempty subset $A$ of $R$ define $A'=\{h'_r\mid r\in A\}.$ If $A$ is an ideal of $R$, then $A'$ is an ideal of $R'$.\par

The ring $(hR,\A)$ of skew Hurwitz polynomials over a ring $R$ is the subring of $(HR,\A)$ that consists elements of the form $f\in (HR,\A)$ with $\Delta(f)<\infty$ (see \cite{M}).\\

The notion of strongly prime rings was introduced by Handelman and
Lawrence in \cite{2}. Since then strongly prime rings have been extensively
studied. F.Cedo in \cite{1} presented an example of a ring $R$ such that $R$ is not
strongly prime but the power series ring $R[[x]]$ is strongly prime on both
sides. It is well-known that the Hurwitz polynomial ring $hR$ is right strongly prime
if and only if the coefficient ring $R$ is right strongly prime. The same argument can be applied to show that an analogous statement holds for skew Hurwitz polynomial rings of automorphism type. The aim of this note is to characterize the left and right strongly primeness of skew Hurwitz polynomial rings of endomorphism type. It is not a surprise that the obtained characterization is not left-right symmetric. This, in turn, enable to construct easy examples of rings which are strongly prime only on one side. The examples known earlier were much more complicated (see \cite{2,4}).
Recall from \cite{2}, that a subset $F \subseteq R$ is a left (right) \textit{insulator} if $F$ is finite and $\ell_R(F) = 0$ ($r_R(F) = 0$). \\
We need the following lemma in the sequel.
\begin{lemma}\label{1.1}\cite{2}
For the ring $R$ the following conditions are equivalent:
\begin{enumerate}
\item
$R$ is left (right) strongly prime.
\item
 Every nonzero ideal of $R$ contains left (right) insulator.
\item
 Every nonzero left (right) ideal of $R$ contains left (right) insulator.
\item
 Every nonzero principal left (right) ideal of $R$ contains left (right)
insulator.
\end{enumerate}
\end{lemma}
Notice that if $(hR,\alpha)$ is left (right) strongly prime, then $\alpha$ has to be a
monomorphism, as otherwise $(hR,\alpha)$ would not be prime.
A left ideal $I$ of $R$ is called a left \textit{$\alpha$-ideal} if $\alpha(I) \subseteq I$.
We say that $R$ is left \textit{$\alpha$-strongly prime} if any nonzero left $\alpha$-ideal $I$ of $R$
contains a finite subset $F$ such that $\ell_R(\alpha^k(F)) = 0$ for any $k \geq 0$.
\begin{remark}\label{1.2}\hfill
\begin{enumerate}
\item
 When $\alpha = id_R$, then $R$ is left $\alpha$-strongly prime if and only if it
is left strongly prime.
\item
 When $\alpha$ is an automorphism, then $k$ in the above definition can be
replace by $0$.
\item
 If $R$ is left $\alpha$-strongly prime, then $\alpha$ has to be a monomorphism,
since $ker(\alpha)$ is an $\alpha$-ideal of $R$.
\end{enumerate}
\end{remark}
\begin{theorem}\label{1.3}
The following conditions are equivalent:
\begin{enumerate}
\item\label{1.3_1}
 $(hR,\alpha)$ is a left strongly prime ring.
\item\label{1.3_2}
 $R$ is a left $\alpha$-strongly prime ring.
 \end{enumerate}
\end{theorem}
\begin{proof}
\eqref{1.3_1}$ \Rightarrow $\eqref{1.3_2}
Let $I$ be a nonzero left $\alpha$-ideal of $R$ and $(hI,\alpha)$ denote
the set of all Hurwitz polynomials from $(hR,\alpha)$ with all coefficients from $I$. Then
$(hI,\alpha)$ is a nonzero left ideal of $(hR,\alpha)$. Thus, by assumption, it contains
a left insulator  $\widehat{F}$. Then $h_{n+1}\widehat{F}$ is also a left insulator for any $k\geq 0$. Let $F$ denote the set of all coefficients of polynomials from $\widehat{F}$. Then, for any
$k \geq 0$, $\alpha^ k(F)$ is the set of all coefficients of polynomials from $h_{n+1}\widehat{F}$. Therefore
$\ell_R(\alpha^ k(F)) = 0$ for any $k \geq 0$. This shows that $R$ is left $\alpha$-strongly
prime.

\eqref{1.3_2}$ \Rightarrow $\eqref{1.3_1}
Let $J$ be a nonzero left ideal of $(hR,\alpha)$ and $I$ denote the set of leading coefficients of all polynomials from $J$. Then $I$ is an $\alpha$-invariant left ideal of $R$. Hence, by assumption, we can find a finite set $F \subseteq I$ such that
$\ell_R(\alpha^k(F)) = 0$ for any $k \geq0$. Let $ \widehat{F}= \{f_{a} \mid  a \in F\}$, where $f_{a}\in J$ denotes a polynomial having the leading coefficient $a$, for $a \in F$. Then it is
standard to check that $ \widehat{F}$ is a left insulator contained in $J$, thus $(hR,\alpha)$ is
left strongly prime.
\end{proof}
The following theorem describes right strongly primeness of $(hR,\alpha)$.
\begin{theorem}\label{1.4}
 For $(hR,\alpha)$ the following conditions are equivalent:
 \begin{enumerate}
 \item\label{1.4_1}
$(hR,\alpha)$ is a right strongly prime ring.
 \item\label{1.4_2}
 \begin{enumerate}
 \item
  $\alpha$ is a monomorphism;
  \item
 for any $0\not= a\in R$ and $m \geq 0$ there exist $k \geq 0$ and a finite
set $F \subseteq a\alpha^ m(R) + \alpha(a)\alpha^{ m+1}(R) +\cdots +\alpha^k(a)\alpha^{ m+k}(R)$ such that
$r_R(F)\cap \alpha^ n(R) = 0$ for some $n \geq 0$.
 \end{enumerate}
 \end{enumerate}
\end{theorem}
\begin{proof}
\eqref{1.4_1}$ \Rightarrow $\eqref{1.4_2}
Clearly $\alpha$ is a monomorphism. Let $0 \not= a \in R$ and
$m \geq 0$. Then  $J = a\alpha^ m(R)h_{m+1}(hR,\alpha)$ is a nonzero right ideal of $ (hR,\alpha) $. Thus,
by assumption, $J$ contains a right insulator $\widehat{F} = \{a_{i,m+k}h_{m+k+1}+\cdots,+a_{i,m}h_{m+1}\in J\mid 1
\leq i \leq s\}$. Define $F =\{ a_{i,m+k}, \alpha(a_{i,m+k-1}),\ldots,\alpha^k(a_{i,m})\mid 1
\leq i \leq s\}$
and set $n = m+k$. Then $F \subseteq a\alpha^m(R)+\alpha(a)\alpha^{ m+1}(R)+\cdots+\alpha^ k(a)\alpha^ n(R)$.
Let $r \in R$ be such that $F\alpha^ n(r) = 0$. Then, since $\alpha$ is a monomorphism,
$a_{i,n}\alpha^ n(r) = 0, a_{i,n-1}\alpha^{n-1}(r) = 0, \ldots, a_{i,m}\alpha^m(r) = 0$ for any $1 \leq i \leq s$.
This means that $r \in r_{(hR,\alpha)}( \widehat{F})$, so $r = 0$ and $\alpha^ n(r) = 0$. Thus,
$r_{R}(F)\cap \alpha^n(R) = 0$ follows.

\eqref{1.4_2}$ \Rightarrow $\eqref{1.4_1}
Let $J$ be a nonzero ideal of 
$(hR, \alpha)$ and $f = h_{m+1}a+\cdots  + h_1a_0\in J$, where $a \not= 0$. By assumption, there exist $k \geq 0$ and a finite subset
$F \subseteq a\alpha^ m(R) +\cdots +\alpha^ k(a)\alpha^{ m+k}(R)$ such that $r_R(F) \cap \alpha^n(R) = 0$ for
some $n$. Without loss of generality, we may assume that $F \subseteq a\alpha^ m(R) \cup\cdots\cup \alpha^ k(a)\alpha^{ m+k}(R)$ and $0 \not\in F$. Let $u = max\{n,m + k\}$. Since $J$ is a
two-sided ideal, for any $b \in F$ we can pick a polynomial $f_b \in J$ of degree
$u$ having $b$ as the leading coefficient. Let $\widehat{F} = \{f_b \mid b\in F\}$. Since $u \geq n$,
$r_R(F) \cap \alpha^u(R) = 0$. This yields easily that $r_{(hR,\alpha)}( \widehat{F}) = 0$ and
shows that $ \widehat{F} $ is a right insulator of $ (hR,\alpha) $. Thus, $(hR,\alpha)  $ is right strongly
prime.
\end{proof}
Notice that when $\alpha$ is an automorphism of $R$ then the property $(b)$ from
the above theorem boils down to: \\``{}every nonzero right $\alpha$-ideal of $R$ contains
right insulator". Thus characterizations obtained in Theorems \ref{1.3}, \ref{1.4} are
symmetric in this case.
The above mentioned theorems enable to construct easy examples of rings
which are strongly prime only on one side.
\begin{example}
Let $K$ be a field and $R = K\langle x_0, x_1, \ldots \mid x_kx_{\ell}=0$ for all
$k \geq \ell \rangle$. Let $\alpha$ denote the $K$-endomorphism of $R$ given by $\alpha(x_k) = x_{k+1}$ for
all $k \geq 0$. Then:
\begin{enumerate}
\item\label{1.5_1}
$(hR,\alpha)$ is not left strongly prime.
\item\label{1.5_2}
$(hR,\alpha)$ is right strongly prime.
\end{enumerate}
\end{example}
\begin{proof}
\eqref{1.5_1}
It is easy to see that for every 
finite subset $F$ of $R$, $\ell_{R}(F) \not=
0$. Thus, by Theorem \ref{1.3}, $(hR,\alpha)$ is not left strongly prime.

\eqref{1.5_2}
Let $0 \not= a \in R$ and $n = 1 + max\{\ell \mid x_{\ell}$ appears in some monomial
from $a\}$. One can easily check that $r_R(a)\cap \alpha ^n(R) = 0$. Then Theorem \ref{1.4} yields that $(hR,\alpha)$ is right strongly prime.
\end{proof}
\bibliographystyle{amsplain}

\end{document}